\documentclass{amsart}

\usepackage{amsmath}
\usepackage{amssymb}
\usepackage{amsthm}
\usepackage{amsfonts}
\usepackage{setspace}
\usepackage{cite}
\usepackage[usenames,dvipsnames,svgnames,table]{xcolor}
\usepackage[colorlinks=true, citecolor=blue, linkcolor=Crimson]{hyperref}

\usepackage{color, etoolbox, xspace}
\newtoggle{draftmode}

\newcommand{\DNR}{\mathrm{DNR}}
\newcommand{\DNC}{\mathrm{DNC}}

\newcommand{\estr}{\langle\rangle}

\newcommand{\bstrings}{\omega^{<\omega}}
\newcommand{\dset}[2]{\{#1 : #2 \}}

\newcommand{\ML}{Martin-L\"{o}f}

\DeclareMathOperator{\upto}{\upharpoonright}
\DeclareMathOperator{\converges}{\downarrow}

\theoremstyle{plain}
\newtheorem{theorem}{Theorem}[section]

\newtheorem{corollary}[theorem]{Corollary}
\newtheorem{lemma}[theorem]{Lemma}
\newtheorem{claim}[theorem]{Claim}

\theoremstyle{definition}
\newtheorem{definition}[theorem]{Definition}

\newcommand{\pP}{\mathbb{P}}

\title[Effective bi-immunity]{Effective bi-immunity and randomness}

\author{Achilles A. Beros}
\address[Achilles A. Beros]{Department of Mathematics\\
University of Hawai`i at M{\=a}noa\\
Honolulu, HI 96822, USA}
\email{beros@math.hawaii.edu}

\author{Mushfeq Khan}
\address[Mushfeq Khan]{Department of Mathematics\\
University of Hawai`i at M{\=a}noa\\
Honolulu, HI 96822, USA}
\email{khan@math.hawaii.edu}

\author{Bj{\o}rn Kjos-Hanssen}
\thanks{This work was partially supported by a grant from the Simons Foundation (\#315188 to Bj{\o}rn Kjos-Hanssen). This material is based upon work supported by the National Science Foundation under Grant No.~1545707.}
\address[Bj{\o}rn Kjos-Hanssen]{Department of Mathematics\\
University of Hawai`i at M{\=a}noa\\
Honolulu, HI 96822, USA}
\email{bjoern.kjos-hanssen@hawaii.edu}

\newcommand{\concat} {\hat{\ }}

\renewcommand{\DNC}{\DNR}

\DeclareMathOperator{\EBI}{EBI}

\DeclareMathOperator{\MLR}{MLR}
\DeclareMathOperator{\Low}{Low}

\theoremstyle{definition}

\begin{document}

	\begin{abstract}
		We study the relationship between randomness and effective bi-immunity. Greenberg and Miller have shown that for any oracle $X$, there are arbitrarily slow-growing $\DNR$ functions relative to $X$ that compute no \ML\ random set. We show that the same holds when \ML\ randomness is replaced with effective bi-immunity. It follows that there are sequences of effective Hausdorff dimension 1 that compute no effectively bi-immune set.

		We also establish an important difference between the two properties. The class $\Low(\MLR, \EBI)$ of oracles relative to which every \ML\ random is effectively bi-immune contains the jump-traceable sets, and is therefore of cardinality continuum.
	\end{abstract}

	\maketitle

	\section{Introduction}
		Let $W_0$, $W_1$, $W_2$, ... be an effective enumeration of the recursively enumerable (or r.e.) sets of natural numbers. An infinite set $A$ of natural numbers is said to be \emph{immune} if it contains no infinite r.e. subset. It is said to be \emph{effectively immune} when there is a recursive function $f$ such that for all $e$, if $W_e$ is a subset of $A$, then $|W_e| \le f(e)$. The interest in sets whose immunity is effectively witnessed in this manner originally arose in the search for a solution to Post's problem.

		The complement of an effectively immune set, if it is r.e., is called \emph{effectively simple}. Smullyan \cite{Smullyan} appears to be the first to explicitly isolate the notion, observing that Post's construction \cite{Post} of a simple set\footnote{A simple set is an r.e.\ set whose complement is immune.} actually produces an effectively simple set. Sacks \cite{SacksSimpleSet} established the existence of a simple set that is not effectively simple. Subsequently, Martin \cite{Martin} showed that every effectively simple set is Turing complete, that is, it computes the halting problem, and thus cannot constitute a solution to Post's problem.

		A key result that establishes the significance of the notion of effective immunity outside the context of the co-r.e. sets is a theorem by Jockusch \cite{Jockusch} that says that the Turing degrees of the effectively immune sets coincide with those of the diagonally nonrecursive (or $\DNR$) functions. Recently, Jockusch and Lewis \cite{jockusch-lewis} have shown that every $\DNR$ function computes a \emph{bi-immune} set, i.e., one such that both it and its complement are immune. They left open the question of whether the result could be extended to show that every $\DNR$ function computes an \emph{effectively bi-immune set}.

		\begin{definition}
		 	A set $X$ is \emph{effectively bi-immune} (or $\EBI$) if $X$ and its complement, $\bar{X}$, are both effectively immune. If $f$ is a recursive function that witnesses the effective immunity of both $X$ and $\bar{X}$, we say it is \emph{effectively bi-immune via $f$}, or $f$-$\EBI$.
		\end{definition}

		The first author has provided a negative answer \cite{beros-dnc} to Jockusch and Lewis's question. To summarize: every $\DNR$ function computes an effectively immune set (in fact, of the same Turing degree), and a bi-immune set, but not every $\DNR$ function computes an $\EBI$ set. In Section \ref{sec:ebi-above-dnr-rec}, we provide a short proof of the main result from \cite{beros-dnc} that builds on previous work by Ambos-Spies, Kjos-Hanssen, Lempp, and Slaman \cite{DNRWWKL}.

		Every \ML\ random set is $\EBI$. How close are these properties? Greenberg and Miller \cite{MillerGreenberg} have shown that there are sets of effective Hausdorff dimension 1 that compute no \ML\ random set. The main result of Section \ref{sec:slow-dnr} shows that there are sets of the former type that compute no $\EBI$ set, and is a possible strengthening of the Greenberg-Miller result.

		It is not known whether every $\EBI$ set computes a \ML\ random set. However, existing results imply that the Turing degrees of these two classes do not coincide. Barmpalias, Lewis, and Ng \cite{BarmpaliasLewisNg} have shown that every PA degree is the join of two \ML\ random degrees. The join of two $\EBI$ sets is easily seen to be $\EBI$, and so $\EBI$ sets are present in every PA degree, in particular, the incomplete ones. Such a degree cannot contain a \ML\ random set, by a theorem of Stephan \cite{Stephan}.

		\section{Computing recursively bounded DNR functions}\label{sec:ebi-above-dnr-rec}

		The following has been obtained independently by Sanjay Jain and Ludovic Patey:

			\begin{theorem}\label{thm:ebi-rec-dnr}
				Every $\EBI$ set computes a recursively bounded $\DNR$ function. Moreover, a recursive bound for the $\DNR$ function can be obtained uniformly from a witnessing function for the $\EBI$ set.
			\end{theorem}
			\begin{proof}
				Let $\gamma$ be any recursive bijection from $\omega$ to the collection of finite subsets of $\omega$, and for an infinite set $Y \subseteq \omega$,
				let $Y_n$ denote the set consisting of the first $n$ elements of $Y$.

				Suppose $X$ is effectively bi-immune via $f$. Let $h$ be a recursive function such that for all $n$,
				\[W_{h(n)} = \begin{cases}
					\gamma(\varphi_n(n)), & \text{ if $\varphi_n(n) \converges$}\\
					\emptyset, & \text{ otherwise.}
				\end{cases}\]

				Now let $g(n) = \gamma^{-1}(X_{f(h(n)) + 1})$ and let $\bar{g}(n) = \gamma^{-1}(\bar{X}_{f(h(n)) + 1})$.
				We claim that both $g$ and $\bar{g}$ are $\DNR$. Suppose that for some $e$, $\varphi_e(e) = g(e)$.
				Then $X_{f(h(e)) + 1} = W_{h(e)}$. But then $W_{h(e)} \subset X$ and $|W_{h(e)}| > f(h(e))$, a contradiction. The argument for $\bar{g}$ is identical.

				Finally, let $\tilde{g} = \min(g, \bar{g})$. Clearly, $\tilde{g}$ is $\DNR$ and recursive in $X$.
				Given any $n \in \omega$, the largest elements in $X_{f(h(n)) + 1}$ and $\bar{X}_{f(h(n)) + 1}$ cannot both be larger than $2f(h(n)) + 1$.  Thus, letting
				\[
					\pi(n) = \max\dset{\gamma^{-1}(D)}{\max(D) \le 2f(h(n)) + 1},
				\]
				we have $\tilde{g}(n) \le \pi(n)$.
			\end{proof}

			Ambos-Spies et al. \cite{DNRWWKL} have shown that there is a $\DNR$ function that computes no recursively bounded $\DNR$ function, and so we reprove the main result of \cite{beros-dnc}:

			\begin{corollary}
				There is a $\DNR$ function that computes no $\EBI$ set.
			\end{corollary}

			It is worth noting that the construction in \cite{beros-dnc} achieves significantly more than was claimed in that paper. It partially relativizes, and Turing reduction can be replaced with recursive enumeration:
			\begin{theorem}[Beros]
				For any set $A$, there is a function $f$ that is $\DNC$ relative to $A$, such that no $\EBI$ set is r.e.\ in $f$.
			\end{theorem}

	\section{Slow-growing DNR functions}\label{sec:slow-dnr}

	In the language of mass problems, Theorem~\ref{thm:ebi-rec-dnr} says that the problem of computing a recursively bounded $\DNR$ is weakly (or Muchnik) reducible to that of computing an $\EBI$ set. One might wonder if the reverse is true, that is, if the two mass problems can be shown to be equivalent. Failing that, one might hope to show that sufficiently slow-growing $\DNR$ functions suffice. More precisely, perhaps there is a slow enough recursive bound $g$ such that all $g$-bounded $\DNR$ functions compute $\EBI$ sets. Khan and Miller have shown \cite{KhanMiller} that by varying $g$, one can obtain a proper hierarchy of mass problems of recursively bounded $\DNR$ functions. Our main result in this section settles these questions.
	\begin{definition}
		An \emph{order function} is a recursive, unbounded and nondecreasing function from $\omega$ to $\omega \setminus \{0, 1\}$.
	\end{definition}

			\begin{theorem} \label{thm:slow-no-ebi}
			For each order function $g$, for each oracle $X$, there is a $g$-bounded function $f$ that is $\DNR$ relative to $X$ and that computes no $\EBI$ set.
			\end{theorem}

			In other words, there are arbitrarily slow-growing $\DNR$ functions relative to any oracle that compute no effectively bi-immune set. On the other hand, sufficiently slow-growing $\DNR$ functions are known to compute sets of effective Hausdorff dimension 1:

			\begin{theorem}[Greenberg and Miller \cite{MillerGreenberg}]
				There is an order function $h$ such that every $h$-bounded $\DNR$ function computes a set of effective Hausdorff dimension 1.
			\end{theorem}

			Together, these theorems imply the following:

			\begin{corollary}
				There is a real of effective Hausdorff dimension 1 that computes no $\EBI$ set.
			\end{corollary}

			In order to prove Theorem~\ref{thm:slow-no-ebi}, we force with bushy trees.

		\subsection{Definitions and combinatorial lemmas}\label{sec:defs-and-combinatorics}

			The following definitions can also be found in \cite{MillerGreenberg} and \cite{KhanMiller}.

			\begin{definition}
				Given $\sigma \in \bstrings$, we say that a tree $T \subseteq \bstrings$ is \emph{$n$-bushy above $\sigma$} if
				every element of $T$ is comparable with $\sigma$, and
				for every $\tau \in T$ that extends $\sigma$ and is not a leaf of $T$, $\tau$ has at least $n$ immediate extensions in $T$.
				We refer to $\sigma$ as the \emph{stem} of $T$.
			\end{definition}

			\begin{definition}
				Given $\sigma \in \bstrings$, we say that a set $B \subseteq \bstrings$ is \emph{$n$-big above $\sigma$} if
				there is a finite $n$-bushy tree $T$ above $\sigma$ such that all its leaves are in $B$.
				If $B$ is not $n$-big above $\sigma$ then we say that $B$ is \emph{$n$-small} above $\sigma$.
			\end{definition}

			Proofs of the following lemmas can be found in \cite{MillerGreenberg} and \cite{KhanMiller}.

			\begin{lemma}[Smallness preservation property]\label{lem:big-subset}
				Suppose that $B$ and $C$ are subsets of $\bstrings$ and that $\sigma \in \bstrings$.
				If $B$ and $C$ are respectively $n$- and $m$-small above $\sigma$, then $B \cup C$ is $(n+m-1)$-small above $\sigma$.
			\end{lemma}

			\begin{lemma}[Small set closure property]\label{lem:small-set-closure}
				Suppose that $B \subset \bstrings$ is $n$-small above $\sigma$. Let $C = \dset{\tau \in \bstrings}{\text{$B$ is $n$-big above $\tau$}}$.
				Then $C$ is $n$-small above $\sigma$.
				Moreover $C$ is \emph{$n$-closed}, meaning that if $C$ is $n$-big above a string $\rho$, then $\rho \in C$.
			\end{lemma}

		\subsection{Proof of Theorem~\ref{thm:slow-no-ebi}}
			For an order function $g$, let $g^{<\omega}$ denote the set of strings in $\bstrings$ whose entries are pointwise bounded by $g$. We define $g^\omega$ analogously.

			We work entirely in $g^{<\omega}$, forcing with conditions of the form $(\sigma, B)$, where $\sigma \in g^{<\omega}$ and $B \subset g^{<\omega}$ and
			$B$ is $g(|\sigma|)$-small above $\sigma$. A condition $(\sigma, B)$ \emph{extends} another condition $(\tau, C)$ if $\sigma\succeq\tau$ and $C \subseteq B$. Let $\pP$ denote this partial order of conditions. Let $[\sigma]$ denote the elements of $g^\omega$ that extend $\sigma$, and let $[B]^\prec$ denote the set of elements of $g^\omega$ that extend an element of $B$.

			For a functional $\Gamma$ and a recursive function $q$, let $\+H_{\Gamma, q}$ be the set of all conditions $(\sigma, B)$ such that
			if $f \in [\sigma] \setminus [B]^\prec$, then $\Gamma^f$ is not effectively bi-immune as witnessed by the function $q$.

			We assume that for all $f \in g^\omega$, for any functional $\Gamma$, the domain of $\Gamma^f$ is an initial segment of $\omega$.

			\begin{lemma}
				$\+H_{\Gamma, q}$ is dense in $\pP$.
			\end{lemma}
			\begin{proof}
				Let $(\sigma, B)$ be any condition and suppose $B$ is $k$-small (and $k$-closed) above $\sigma$.
				By suitably extending $\sigma$, we may assume that $g(|\sigma|) \ge 8k$.

				Suppose first that there is a $\tau \notin B$ extending $\sigma$ such that for some $m$,
				\[
					C_m = \dset{\rho}{\Gamma^\rho(m)\converges}
				\]
				is $7k$-small above $\tau$.
				Then $(\tau, B \cup C_m)$ is a condition extending $(\sigma, B)$, and for every $f \in [\sigma] \setminus [B]^\prec$, $\Gamma^f$ is not total.
				So we assume from now on that for every $\tau \notin B$ extending $\sigma$ and every $m \in \omega$, $C_m$ is $7k$-big above $\tau$.

				It now follows that for every $\tau \notin B$ extending $\sigma$,
				there is an infinite exactly $6k$-bushy tree $T_\tau$ without leaves above $\tau$ such that for every $f \in [T_\tau]$, $\Gamma^f$ is total:
				Let $S_0$ consist of $\tau$ and its initial segments.
				Next, suppose we have already constructed a finite tree $S_n$ that is exactly $6k$-bushy above $\tau$ and such that
				for each leaf $\rho$ of this tree, $\rho \notin B$ and $\Gamma^\rho$ is defined up to $n-1$.
				By our assumption above, $C_{n}$ is $7k$-big above each leaf, so $C_{n}\setminus B$ is $6k$-big above each leaf by Lemma~\ref{lem:big-subset}.
				For a leaf $\rho$ of $S_n$, let $A_\rho$ be a finite exactly $6k$-bushy tree above $\tau$ with leaves in $C_n \setminus B$.
				We construct $S_{n+1}$ by appending $A_\rho$ to each leaf $\rho$ of $S_n$. Finally, let $T_\tau = \bigcup_{n \in \omega} S_n$.

				\begin{definition}
					Let $\tau$ be any extension of $\sigma$ that is not in $B$.
					We say $\tau$ \emph{admits fusion} if for infinitely many $m \in \omega$, for some $i \in \{0, 1\}$,
					\[
						\Delta_{\tau, m, i} = \dset{\rho \in T_\tau}{\Gamma^\rho(m) = i}
					\]
					is $4k$-big above $\tau$.
				\end{definition}

				\begin{claim}\label{clm:fusion}
					If $\tau$ admits fusion, then there is a subtree $T'$ of $T_\tau$ which is $2k$-bushy above $\tau$ and
					for infinitely many $m \in \omega$ there is an $i \in \{0, 1\}$ with $\Gamma^f(m) = i$ for all $f \in [T']$.
				\end{claim}
				\begin{proof}
					Let $I_0 \subseteq \omega$ be such that for all $l \in I_0$, either $\Delta_{\tau, l, 0}$ or $\Delta_{\tau, l, 1}$ is $4k$-big above $\tau$, and
					let $\Delta_l$ denote whichever one is. Let $S_0$ consist of $\tau$ and its initial segments, and note that $S_0$ is $2k$-bushy above $\tau$.

					Next, suppose that we have constructed a finite tree $S_k \subseteq T_\tau$, $2k$-bushy above $\tau$, and a subset $I_k$ of $\omega$ such that:
					\begin{enumerate}
						\item There are $n_0 < n_1 < \dots < n_{k-1}$ such that for each $i < k$, $\Gamma^\rho(n_i)$ is constant as $\rho$ ranges over the leaves of $S_k$.
						\item For all $l \in I_k$, there is a tree which is $4k$-bushy above $\tau$ and contains $S_k$, whose leaves are in $\Delta_l$.
					\end{enumerate}

					Let $n$ be the least element in $I_k$ greater than $n_{k - 1}$, and
					let $C$ be a finite $4k$-bushy tree above $\tau$ containing $S_k$ whose leaves are in $\Delta_n$.
					Now for any $l > n$ in $I_k$, if $F_l$ is any $4k$-bushy tree above $\tau$ containing $S_k$ with leaves in $\Delta_l$, then
					$F_l \cap C$ is a $2k$-bushy tree above $\tau$ that contains $S_k$. To see this, let $\rho \in F_l \cap C$.
					If $\rho$ has an immediate extension in $F_l \cap C$, then $\rho$ has $4k$ many extensions in each of $F_l$ and $C$.
					But $T_\tau$ is exactly $6k$-bushy above $\tau$, so at least $2k$ of these must be in $F_l \cap C$.

					It follows from the pigeonhole principle (note that $C$ is finite) that there is an infinite subset $I_{k+1}$ of $I_k$, such that for all $l \in I_{k+1}$,
					there are $4k$-bushy subtrees above $\tau$ with leaves in $\Delta_l$ that
					intersect $C$ in the \emph{same} $2k$-bushy subtree $S_{k+1}$ above $\tau$ that contains $S_k$.

					This completes the definition of the sequence $\langle S_k \rangle_{k \in \omega}$. Let $T' = \bigcup_{k \in \omega} S_k$. Then $T'$ is as desired.
				\end{proof}

				\paragraph{Case 1: Some $\tau \succeq \sigma$ admits fusion.}
				We begin by extending $\sigma$ to $\tau$ obtaining the condition $(\tau, B)$ (note that $\tau$ is by definition not in $B$).
				Claim~\ref{clm:fusion} implies that, uniformly in $k$, we can find
				a finite $2k$-bushy tree $R_k$ above $\tau$ such that for at least $k$ distinct inputs $m$,
				$\Gamma^\rho(m)$ is constant as $\rho$ ranges over the leaves of $R_k$.

				For $i \in \{0, 1\}$, let $W_{e_i}$ be the r.e. set defined as follows:
				If $m = \max(q(e_0), q(e_1))$\footnote{We use the recursion theorem here.}, let
				\[
					W_{e_i} = \dset{n}{\Gamma^\rho(n) = i \text{ for each leaf $\rho$ of $R_{2m + 1}$}}.
				\]
				It must now be the case that for some $i \in \{0, 1\}$, $|W_{e_i}| > m$. Suppose $i = 0$ (the argument for the other case is symmetric).
				Let $\rho \succeq \tau$ be a string in $R_{2m +1} \setminus B$ (note that $B$ is $k$-small above $\tau$).
				Then $(\rho, B)$ is a condition, and for all $f \in [\rho] \setminus [B]^\prec$, if $\Gamma^f$ is total, then $W_{e_0}$ is contained in its complement.

				\paragraph{Case 2: No extension of $\sigma$ admits fusion.}
				This means that for every extension $\tau$ of $\sigma$ such that $\tau \notin B$, there is an $m_\tau \in \omega$ such that
				for all $l \ge m_\tau$, both $\Delta_{\tau, l, 0}$ and $\Delta_{\tau, l, 1}$ are $4k$-small above $\tau$. Recall that $T_\tau$ is exactly $6k$-bushy above $\tau$. Therefore, $\Delta_{\tau, l, 0} \cup \Delta_{\tau, l, 1}$ is $6k$-big above $\tau$. By Lemma~\ref{lem:big-subset}, if one of these sets is $2k$-small above $\tau$, the other is $4k$-big, so both must be $2k$-big above $\tau$.

				Let $S_0$ consist of $\sigma$ and its initial segments. Note that no leaf of $S_0$ is in $B$.

				Proceeding by induction, suppose we have constructed a finite $k$-bushy tree $S_k \subseteq T_\sigma$ above $\sigma$ with the following properties:
				\begin{enumerate}
				 	\item There are $n_0 < n_1 <  \dots < n_{k-1}$ such that for every leaf $\rho$ of $S_k$, $\Gamma^\rho(n_i) = 0$ for each $i < k$.
				 	\item None of the leaves of $S_k$ is in $B$.
				 \end{enumerate}

				Let $n_k = \max\dset{m_\tau}{\text{$\tau$ a leaf of $S_k$}} + 1$.
				By the observation above, for each leaf $\tau$ of $S_k$, $\Delta_{\tau, n_k, 0}$ is $2k$-big above $\tau$, so
				$\Delta_{\tau, n_k, 0} \setminus B$ is $k$-big above $\tau$. Let $F_\tau$ be a finite $k$-bushy tree above $\tau$ with leaves in
				$\Delta_{\tau, n_k, 0} \setminus B$, and let $S_{k+1}$ be obtained from $S_k$ by extending each leaf $\tau$ of $S_k$ by $F_\tau$.

				Finally, let $T' = \bigcup_{k \in \omega} S_k$. Then for all $k$, for every $f \in [T']$, $\Gamma^f(n_k) = 0$.
				A strategy similar to the one employed in case 1 now diagonalizes against the pair $(\Gamma, q)$. This concludes the proof of the lemma.
			\end{proof}

			To conclude the proof of Theorem \ref{thm:slow-no-ebi},
			let $B_{\DNR^X}$ be the set of finite strings that cannot be extended to a $\DNR$ relative to $X$ and
			let $\+G$ be any filter containing $(\estr, B_{\DNR^X})$ that meets $\+H_{\Gamma, q}$ for each functional $\Gamma$ and recursive function $q$.
			Then $f_\+G$ is a $g$-bounded $\DNR$ relative to $X$ and does not compute an effectively bi-immune set.

	\section{Traceability and lowness}
		There is more than one way to define effective immunity relative to an oracle. We focus on a partial relativization, motivated by the fact that under this definition, a \ML\ random set relative to any oracle $X$ will be effectively immune relative to $X$ via the function $h(e)=e+c$ for some $c \in \omega$.
		\begin{definition}
			An infinite set $R$ is \emph{effectively immune relative to $G$} if there is a recursive function $h$ such that for all $e$, if $W_e^G\subseteq R$ then $|W^G_e|\le h(e)$.
		\end{definition}
		\begin{definition}
			A set $G\in\Low(\MLR,\EBI)$ if each $\MLR$ set $R$ is $\EBI$ relative to $G$.
		\end{definition}
		\begin{definition}
			A recursive enumerable (r.e.) \emph{trace} $T$ is a sequence of sets $T^{[e]}=W_{g(e)}$, $e\in\omega$
			such that $|W_{g(e)}|\le h(e)$ for all $e$, where $g$ and $h$ are recursive functions.
			For a function $f$, we say that $T$ \emph{traces} $f$ on input $n$ if $f(n)\in T^{[n]}$.
			A set $G$ is \emph{jump traceable} if there is a r.e.~trace $T$ such that for all $e$, if $\varphi_e^G(e)\downarrow$ then
			$\varphi^G_e(e)\in T^{[e]}$.
		\end{definition}
		Theorem \ref{bjoern} gives a contrast between $\MLR$ and $\EBI$.

	\begin{theorem}\label{bjoern}
		Each jump traceable Turing degree is $\Low(\MLR,\EBI)$.
	\end{theorem}
	\begin{proof}
		Let $G$ be jump traceable via $h$, and let $J^G$ denote the diagonal partial recursive function relative to $G$.

		We define a recursive function $f$ knowing its index in advance by the recursion theorem. Let $\varphi$ be the function partial recursive in $G$ that on input $e$, waits for $W^G_e$ to enumerate at least $f(e) + 1$ elements, and then outputs the natural number that encodes the finite set $B_e$ consisting of the first $f(e) + 1$ of these. Next, let $p$ be a recursive function such that $J^G \circ p = \varphi$. Note that $p$ can be obtained uniformly from an index for $f$. Now define $f$ so that
		\[h(p(e))\, 2^{- (f(e) + 1)} \le 2^{-e}.\]

		We have an r.e. trace $S^{[p(e)]}$ for (the code for) $B_e$, and there are at most $h(p(e))$ many elements in it. Let $B_e^{(i)}$ denote the $i$th candidate for $B_e$ if it exists, for $i<h(p(e))$.
		Then let
		\[
			U_c = \{
				A: (\exists e>c)(\exists i<h(p(e)))(B_e^{(i)}\downarrow\subseteq A)
			\}
		\]
		Then
		\[
		 	\mu(U_c)\le\sum_{e>c} h(p(e))\, 2^{-(f(e)+1)}\le \sum_{e>c}2^{-e}=2^{-c}.
		\]
		If $A$ is $\MLR$ then there exists $c$ such that for all $e>c$ and $i$, it is not the case that $B_e^{(i)}\downarrow\subseteq A$.
		Thus for all $e\ge c$, if $W_e^G\subseteq A$ then $W_e^G$ has size at most $f(e)$.
		Thus $A$ is EBI relative to $G$.	\end{proof}

\section{Canonical immunity}
		It is natural to next consider lowness notions associated with Schnorr randomness. This idea leads us to a new notion of immunity.

		A \emph{canonical numbering of the finite sets} is a surjective function $D:\omega\rightarrow \{A:A\subseteq\omega\text{ and $A$ is finite}\}$ such that
		$\{(e,x): x\in D(e)\}$ is recursive and the cardinality function $e\mapsto |D(e)|$, or equivalently, $e\mapsto\max D(e)$, is also recursive.
		We write $D_e=D(e)$.
		\begin{definition}
			$R$ is \emph{canonically immune} if $R$ is infinite and
			there is a recursive function $h$ such that for each canonical
			numbering of the finite sets $D_e$, $e\in\omega$,
			we have that for all but finitely many $e$, if $D_e\subseteq R$
			then $|D_e|\le h(e)$.
		\end{definition}
		\begin{theorem}\label{lateApril}
			Schnorr randoms are canonically immune.
		\end{theorem}
		\begin{proof}
			Fix a canonical numbering of the finite sets, $\{D_e\}_{e \in \omega}$.
			Define $U_c = \{ X : (\exists e > c) \big( |D_e| \geq 2e \wedge D_e \subset X \big) \}$.
			Since $e \mapsto |D_e|$ is recursive, $\mu(U_c)$ is recursive and bounded by $2^{-c}$.
			Thus, the sequence $\{U_c\}_{c \in \omega}$ is a Schnorr test.
			If $A$ is a Schnorr random, then $A \in U_c$ for only finitely many $c \in \omega$.
			We conclude that $A$ is canonically immune.
		\end{proof}
		\begin{theorem}
			Each canonically immune set is immune.
		\end{theorem}
		\begin{proof}
			Suppose $A$ has an infinite recursive subset $R$. Let $h$ be any recursive function.
			Let $R_n$ denote the set of the first $n$ elements of $R$, and let $\{D_e: e \in \omega\}$ be a canonical numbering of the finite sets such that $D_{2n} = R_{h(2n)+1}$ for all $n \in \omega$.
			For all $n$, $D_{2n}\subseteq R\subseteq A$ and $|D_{2n}| = h(2n)+1 > h(2n)$, and so $h$ does not witness the canonical immunity of $A$.
		\end{proof}

		We now show that canonically immune is the ``correct'' analogue
		of effectively immune.

		\begin{definition}[Kjos-Hanssen, Merkle, and Stephan \cite{MR2813422}]
			A function is strongly nonrecursive (SNR) if it differs from
			each recursive function on all but finitely many inputs.
		\end{definition}

		\begin{theorem}\label{may3}
			Each canonically immune set computes a strongly nonrecursive
			function.
		\end{theorem}
		\begin{proof}
			Let $R$ be canonically immune as witnessed by the recursive function $h$. Define $f(e)$ to be (a code for) the first $h(2e) + 1$ many elements of $R$, and note that $f$ is recursive in $R$. We claim that $f$ is strongly nonrecursive.

			Suppose that the recursive function $g$ is infinitely often equal to $f$. Let $\{D_e : e \in \omega\}$ be any canonical numbering of finite sets such that for all $e$, $D_{2e}$ is the finite set coded by $g(e)$. We now have that for infinitely many $e$, $D_{2e}$ is the set consisting of the first $h(2e) + 1$ many elements of $R$, a contradiction.
		\end{proof}

		Interestingly, Theorem \ref{may3} shows that we can strengthen
		``$D_e\subseteq R$''
		to
		``$D_e$ is an initial segment of $R$''.

		\begin{corollary}
			The following are equivalent for an oracle $A$:
			\begin{enumerate}
				\item $A$ computes a canonically immune set,
				\item $A$ computes an SNR function,
				\item $A$ computes an infinite subset of a Schnorr random.
			\end{enumerate}
		\end{corollary}
		\begin{proof}
			(1) implies (2) is proved in Theorem \ref{may3}.
			(2) implies (3) follows from older results:
			each SNR either is high or computes a DNR \cite{MR2813422},
			hence either computes a Schnorr random \cite{MR2140044} or computes an infinite subset of an
			MLR \cite{MR2518817}, hence either way computes an infinite subset of a Schnorr random.
			(3) implies (1) is proved in Theorem \ref{lateApril}.
		\end{proof}

\section{A class between EI and EBI}

			\begin{theorem}
				There is a bi-immune set such that it is effectively immune while its complement is not.
			\end{theorem}
			\begin{proof}
				We build a set $A$ in stages by describing its characteristic function, $g$.

				\textbf{Stage 0:}  Define $g_0$ to be the function with empty domain.

				\textbf{Stage $\mathbf{2e + 1}$:}  Define $g_{2e+1} \upharpoonright \mbox{dom}(g_{2e}) = g_{2e}$.
				Let $m = min(\mathbb N \setminus \mbox{dom}(g_e))$ and set $g_{2e+1}(m) = 1$.
				If $|W_e| > 2e+1$ and there is no $a \in W_e \cap \mbox{dom}(g_{2e})$ such that $g_{2e}(a) = 0$,
				pick $x \in W_e \setminus \mbox{dom}(g_{2e})$ and set $g_{2e+1}(x) = 0$.
				If $W_e$ is infinite, select a $y \neq x$ such that $y \in W_e \setminus \mbox{dom}(g_e)$ and set $g_{2e+1}(y) = 1$.

				\textbf{Stage $\mathbf{2e + 2}$:}  Define $g_{2e+1} \upharpoonright \mbox{dom}(g_{2e}) = g_{2e}$.
				If $\phi_e$ is total, pick an r.e.~set $W_a$ such that $\phi_e(a) < |W_a| < \infty$ and $|W_a| \cap \mbox{dom}(g_{2e+1}) = \emptyset$.
				Set $g_{2e+2}(x) = 0$ for all $x \in W_a$.

				Notice that there are no more than $2s$ elements $x$ such that $g_s(x) = 1$. So either it is possible to pick an $x$ as in the odd stages, $2e+1$, whenever $|W_e| > 2e+1$, or
				there is already an element of the domain of $g_{2e}$ which is in $W_e$ on which $g_{2e}$ takes the value 0.
				Let $g = \bigcup_{s \in \mathbb N} g_s$.
				Observe that $g$ is total and $\{0,1\}$-valued.
				Let $A$ be the set whose characteristic function is $g$.
				The effective immunity of $A$ is witnessed by $f(x) = 2x+1$ and $\overline{A}$ is clearly immune, however,
				for any total function $h$ there is an r.e.~set $W_a$ such that $h(a) < |W_a|$ and $W_a \subseteq \overline{A}$.
				Thus, $\overline{A}$ is not effectively immune.
			\end{proof}

		\section{Boldface complexity}

		\begin{theorem}\label{thm:fei-closed}
		Let $f$ be a recursive function. The class of reals that are effectively immune via $f$ is closed.
		\end{theorem}

		\begin{proof}
		Suppose that $A$ is not effectively immune via $f$. Then there is a $e$ such that $W_e$ is a subset of $A$ and $|W_e| > f(e)$. If $W_e$ is a finite set, then there is an initial segment $\sigma$ of $A$ such that $W_e$ is contained in any set whose characteristic function extends $\sigma$, and so no extension of $\sigma$ is effectively immune via $f$. So suppose that $W_e$ is infinite. By the recursion theorem, there exists an $e'$ such that $W_{e'}$ consists of the first $f(e') + 1$ elements of $W_e$. Thus, in this case there is also an initial segment $\sigma$ of $A$ such that any set whose characteristic function extends $\sigma$ contains $W_{e'}$, and is therefore not effectively immune via $f$.
		\end{proof}

		Recall that a set of reals is $F_\sigma$ if it is a countable union of closed sets.

		\begin{corollary}
		The class of $\EBI$ reals is $F_\sigma$.
		\end{corollary}

		Additionally, the class is no simpler:

		\begin{theorem}
		The class of $\EBI$ reals is Wadge complete for $F_{\sigma}$.
		\end{theorem}

		\begin{proof}
		Let $A \subset 2^\omega$ be the set of reals that are eventually zero. It is well-known that $A$ is Wadge complete for the $F_\sigma$ sets (see, for example, \cite{Kechris}, Exercise 21.17). We construct a continuous $h: 2^\omega \rightarrow 2^\omega$ such that $X \in A$ iff $h(X)$ is $\EBI$, showing that $A$ is Wadge reducible to the class of $\EBI$ reals.

		We first define a function $f: 2^{<\omega} \rightarrow 2^{<\omega}$ recursively. Let $e_0, e_1, \ldots$ be an increasing list of all codes for total functions.  Also, for each $\tau \in 2^{<\omega}$, let $g_{\tau}$ be an $\EBI$ real which has $\tau$ as an initial segment and let $\alpha_{\tau}^i$ be an extension of $\tau$ such that no real extending $\alpha_{\tau}^i$ is $\phi_{e_i}$-$\EBI$. Note that $\alpha_{\tau}^i$ exists by the argument in the proof of Theorem~\ref{thm:fei-closed}.

		Let $f(\estr) = \estr$.  Suppose $\sigma \in 2^{<\omega}$, let $n = |\sigma|$ and $k$ be the number of bits of $\sigma$ which are $1$.  Given $f(\sigma) = \tau$, we define $f(\sigma \concat 0) = g_{\tau} \upto (|\tau|+1)$ and $f(\sigma \concat 1) = \alpha_{\tau}^k$.  Finally, define $h$ so that the $n$th bit of $h(x)$ is the $n$th bit of $f(x\upto (n+1))$.
		\end{proof}

		\section{Acknowledgements}
		The authors would like to thank Uri Andrews, Daniel Turetsky, Linda Westrick, Rohit Nagpal, and Ashutosh Kumar for helpful discussions.

	\bibliographystyle{plain}

\end{document}